\documentclass{amsart}
\usepackage{amsmath,amssymb,amsfonts}

\newtheorem{theorem}{Theorem}

\newtheorem{corollary}{Corollary}
\newtheorem{proposition}{Proposition}

\theoremstyle{definition}

\newtheorem{example}{Example}

\theoremstyle{remark}

\newcommand{\De}{{\Delta}}

\newcommand{\na}{{\nabla}}

\newcommand{\ep}{{\epsilon}}

\newcommand{\di}{{\mathrm{div}}}

\title[On an inclusion of the essential spectrum of Laplacians]
{On an inclusion of the essential spectrum of Laplacians
under non-compact change of metric}

\begin{document}
\author{Jun Masamune}

\email{jum35@psu.edu} 


\begin{abstract}
It is shown the stability of the essential self-adjointness,
and an inclusion of the essential spectra of Laplacians 
under the change of Riemannian metric on a subset $ K$ of $M$.
The set $K$ may have infinite volume measured 
with the new metric and its completion 
may contain a singular set such as fractal,
to which the metric is not extendable.
\end{abstract}

\date{}
\maketitle{}

\section{Introduction}
Let $(M,g)$ be a connected smooth Riemannian manifold without boundary.
The Laplacian $\Delta$ of $g$ is called essentially self-adjoint if it has the
unique self-adjoint extension $\overline{\Delta}$.
In \cite{Furutani.80}, Furutani showed that if $\De$
with the domain $C^\infty_0(M)$ is essentially self-adjoint and, 
if $g$ is changed on a compact set $K \subset M$ 
to another smooth metric $g'$ on $M$, then
the Laplacian $\De'$ of $g'$ with the domain $C^\infty_0(M)$ is 
essentially self-adjoint; and the essential spectrum are stable under this change.
In particular, the second result
forms a strong contrast to the behavior of the eigenvalues,
since eigenvalues change continuously with the 
perturbation of the metric in a certain way
(see for e.g \cite{BandoUrakawa.83}).

Needless to say, there are many important Riemannian manifolds with singularity,
by which, we mean that $g$ does not extend to the Cauchy boundary
(the difference between the completion of $M$ and $M$); 
such as algebraic varieties, cone manifolds, edge manifolds, 
Riemannian orbifolds.
In general, the analysis on such a singular space is complicated, and
one of the methods to overcome the difficulties is to modify the singularity to 
a simpler one by the perturbation of the Riemannian metric.
The crucial steps in this process is to study the stability 
of the essential self-adjointness of the Laplacian, 
and to understand the behavior of its spectral structure under the perturbation.

Motivated by these facts, we extend Furutani's theorem to more general $K$
so that $K$ is not compact and
its completion $\overline{K}$ includes the singular set.
In this setting, the natural domain $D(\De)$
for the Laplacian is the following:
\begin{equation} \label{equation;1}
\begin{cases}
D(\na) = \{ u \in C^\infty \cap L^2: \na u \in L^2 \}, \\
D(\di) = \{ X \in C^\infty \cap L^2: \di X \in L^2 \}, \\
D(\De) =\{ u \in D(\na): \na u \in D(\di) \}.
\end{cases}
\end{equation}
(We suppress $M$ and the Riemann measure $d\mu_g$
for the sake of simplicity.)
Indeed, if the Cauchy boundary
$\partial_C M$ is almost polar; namely,
\[ {\rm{Cap}} (\partial_C M) = 0, \]
(see Section \ref{Proof of Theorem} for the definition.
See also for e.g.\ \cite{FOT})
then $M$ has negligible boundary \cite{Masamune.05},
and by the Gaffney theorem \cite{Gaffney.55},
$\De$ is essentially self-adjoint.
All through the article, we assume that the Laplacians
have the domain defined in (\ref{equation;1}).
The following is our main result:
\begin{theorem} \label{Th;1}
Let $g$ and $g'$ be Riemannian metrics on $M$ such that
$g=g'$ outside a subset $K$ of $M$. 
If $\De$ is essentially self-adjoint in $L^2$ and the Cauchy boundary
of $K$ with respect to $g'$ is almost polar, 
then $\De'$ is essentially self-adjoint in $L^2(M;d\mu_{g'})$.
Additionally, if there is a function $\chi$ on $M$ satisfying
\begin{equation} \label{cond1}
\na \chi \in L^\infty,\ \De \chi \in L^\infty, \mbox{ and } \chi|_K=1,
\end{equation}
where $\nabla$ is the gradient of $g$, and the inclusion
\begin{equation} \label{cond2}
H^1_0(N;\,d\mu_g) \subset L^2(N;\,d\mu_g) \mbox{ is compact}
\end{equation}
 for some $N \supset N(\mbox{supp}(\chi);\ep)$ with some $\ep>0$,
 where $N(\mbox{supp}(\chi);\ep)$ is the $\ep$-neighborhood 
 of the support of $\chi$, then
 \[ \sigma_{\mbox{ess}}(\overline{\De})  \subset \sigma_{\mbox{ess}}(\overline{\De'}).\]
\end{theorem}
A special case of Theorem \ref{Th;1} is
\begin{corollary}[Furutani's  stability result \cite{Furutani.80}] \label{corollary;1}
Let $g$ and $g'$ be Riemannian metrics on $M$ such that
$g=g'$ outside a compact subset $K$ of $M$. 
If $\De$ is essentially self-adjoint in $L^2$ 
then $\De'$ is essentially self-adjoint in $L^2(M;d\mu_{g'})$, and
\begin{equation} \label{cond3}
\sigma_{\mbox{ess}}(\overline{\De})  = \sigma_{\mbox{ess}}(\overline{\De'}).
\end{equation}
\end{corollary}

A typical example of manifolds which satisfies the condition of Theorem \ref{Th;1} is
given as follows:
\begin{corollary}[See Section \ref{Examples}]
Let $M$ be a complete manifold and 
$\Sigma \subset M$, an almost polar compact subset.
If $\Sigma$ is almost polar with respect to a metric $g'$ on $M \setminus \Sigma$
and $g=g'$ outside a compact set $K \subset M$, then the same conclusion
in the theorem holds true.
\end{corollary}
We may apply Theorem \ref{Th;1} for singular manifolds:
we change $g$ to $g'$ on a bounded set $K \supset \partial_C M$
so that $g'$ can be extended to the almost polar Cauchy boundary
with respect to $g$,
and conclude that $\De$ is essentially self-adjoint
in $L^2(M; d\mu_g)$ and 
$ \sigma_{\rm{ess}} (\overline{\De}) \subset \sigma_{\rm{ess}} (\overline{\De'}) $.

A sufficient condition for $\partial_C M$ to be almost polar
is that it has Minkowski co-dimension greater than 2 \cite{Masamune.99}
(if the metric of $g$ extends to $\partial_C M$ and $\partial_C M$
is a manifold, then it is almost polar
if $\partial_C M$ has co-dimension 2).

The idea to prove the inclusion of the essential spectrum in Theorem \ref{Th;1} is to
apply Weyl's criteria:
a number $\lambda$ belongs to $\sigma_{ess} (\overline{\Delta})$ if and only if
there is a sequence $\phi_n$ of ``limit-eigenfunctions" of $\overline{\Delta}$
corresponding to $\lambda$
(see Proposition \ref{Weyl's criterion} for details). 
Indeed, we show that if $\chi$ satisfies (\ref{cond1}) and (\ref{cond2}), 
then there is a subsequence $\phi_{n(k)}$ such that $(1-\chi) \phi_{n(k)}$
is a limit-eigenfunction of $\overline{\Delta}'$.

Our results differ from Furutani's original results in the following two points.
In order to explain those differences, let us employ an example.
Let $M$ be
\[
M = K \cup B(1),
\]
where
$K=\{ (x,y,z) \in S^2:  z \ge 0 \}   \setminus (0,0,1)$
and 
$B(r)=\{ (x,y,z) \in \mathbb{R}^3: x^2 + y^2 \le r^2,\  z = 0 \}$.
Namely,  $M$ is the $S^2$ with flat bottom with deleted the north point.
 (To be more precise, we need to smooth the 
intersection of $K$ and $B(1)$ so that $M$ is a smooth
Riemannian manifold.)
Since the Cauchy boundary of $M$ is the north point and it has null capacity, 
the Laplacian $\Delta$ is essentially self-adjoint.
We modify $M$ by the stereographic projection so that
$(M, g')$ is the 2-dimensional Euclidean space $\mathbb{R}^2$.
Since $(M, g')$ is complete, the Cauchy boundary is empty, 
and the Laplacian $\Delta'$ is essentially self-adjoint.
Next, we find the function $\chi$ which satisfies condition (\ref{cond1}) as follows:
\[\chi \in C^\infty_0 (M \setminus B(1/3)) \mbox{ and } \chi=1 \mbox{ on $K$}. \]
By letting $\ep=1/3$ and $N=M \setminus B(1/4)$, condition (\ref{cond2}) is satisfied.
Indeed, since the north point has null capacity,
the spectrum of the Laplacian on $M$ consists only of the eigenvalues with finite multiplicity,
whereas $\De'$ has only essential spectrum.
This proves that the inclusion in Theorem  \ref{Th;1} holds.
This example also shows that the assumptions in 
Corollary \ref{corollary;1} is sharp in the
sense that we may not drop the assumption such that $K$ is compact
to obtain (\ref{cond3}); that is, Furutani's stability result.
Indeed, if we modify the metric of $\mathbb{R}^2$ to obtain $M$,
then there is no subset $N$ of $\mathbb{R}^2$ which satisfies condition (\ref{cond2}).

The second difference is that the essential self-adjointness
of $\De$ does not need to imply that of $\De$ restricted to $C^\infty_0(M)$;
for instance, $\Delta$ on $M$ is essentially self-adjoint, 
but $\De$ restricted to $C^\infty_0(M)$
has infinitely many self-adjoint extensions, in particular,
it is not essentially self-adjoint (see for e.g.\ \cite{C}).

We organize the article in the following manner:
In Section \ref{Proof of Theorem}, we prove Theorem \ref{Th;1},
and in Section \ref{Examples}, we present the examples.

\section{Proofs}
\label{Proof of Theorem}
In this section we recall some definitions and prove Theorem \ref{Th;1}
and Corollary \ref{corollary;1}.
For the sake of the simplicity, we often suppress 
the symbols $M$ and $d\mu_g$.

We denote by $(\cdot ,\cdot)$ and $(\cdot ,\cdot)_{1}$
the inner product in $L^2$ and the Sobolev space $H^1$ of order $(1,2)$,
respectively. 
$H^1_0$ is the completion of the set $C^\infty_0$ of smooth functions 
with compact support with respect to the norm
$\| \cdot \|_1 = \sqrt{ (\cdot ,\cdot)_{1} }$.
Let $\mathcal{O}$ be the family of all open subsets of $\overline{M}$.
For $A \in \mathcal{O}$ we define $\mathcal{L}_A 
= \{ u \in H^1: u  \ge 1\ \mu_g\mbox{-a.e.\ on $ M \cap A$} \}$,
\[
\mbox{Cap}(A) =
\begin{cases}
\inf_{u \in \mathcal{L}_A} \| u \|_1, & \mathcal{L}_A \neq \phi, \\
\infty, &\mathcal{L}_A = \phi,
\end{cases}
\]
and
\[ \mbox{Cap}(\partial_C M) = \inf_{A \in \mathcal{O},\ \partial_C M \subset A} \mbox{Cap}(A). \]
We will use 
\begin{proposition}[Lemma 2.1.1 \cite{FOT}]
If $\mathcal{L}_A \neq \phi$ for $A \in \mathcal{O}$, 
there exists a unique element $e_A \in \mathcal{L}_A$ 
called the \emph{equilibrium potential} of $A$ such that\\
$(\rm{i})$ $\| e_A \|^2_1 = \mbox{Cap}(A)$. \\
$(\rm{ii})$ $0 \le e_A \le 1$ $\mu_g$-a.e.\ and $e_A=1$ $\mu_g$-a.e.\ on $A \cap M$. \\
$(\rm{iii})$ If $A,B \in \mathcal O$, $A \subset B$, then $e_A \le e_B$ $\mu_g$-a.e. 
\end{proposition}


We prove Theorem  \ref{Th;1}. We start from:
\begin{proof}[Proof of the essential self-adjointness of $\De'$]
For arbitrary $u \in H^1(M;d\mu_{g'})$, we have to find $\hat{u}_n \in H^1(M;d\mu_{g'})$
which converges to $u$ in $H^1(M;d\mu_{g'})$.
Indeed, this implies that $(M,g')$ has negligible boundary, and hence,
$\De'$ is essentially self-adjoint by Gaffney theorem \cite{Gaffney.55}.

Since $L^\infty(M) \cap H^1(M;d\mu_{g'})$ is dense in $H^1(M;d\mu_{g'})$,
we may assume that $u \in L^\infty(M)$ without the loss of generality.
Let
\begin{equation} \label{eq;3}
\psi := (1-r)_+,
\end{equation}
where $r$ is the distance from $K$.
The function $\psi \in L^\infty (M)$ enjoys the property:
\[ \psi|_{K}=1 \mbox{ and } \| \na \psi \|_{L^\infty} \le 1. \]
Since
\[
|(1-\psi)u (x)| \le (1 + \| \psi \|_{L^\infty} ) |u(x)| \]
and
\[ | \nabla ((1-\psi)u (x))| \le (1 + \| \psi \|_{L^\infty} ) |\nabla u (x)| + |u (x)|,
\]
for almost every $x \in M$, it follows that $(1-\psi)u \in H^1$.
Recalling that the essential self-adjointness of $\De$ implies
$H^1_0= H^1$ \cite{Masamune.05},
we find $v_n \in C^\infty_0(M \setminus K)$
such that
\begin{equation*}
v_n \to  (1-\psi) u \mbox{ as $n\to\infty$}
\end{equation*}
in $H^1(M \setminus K)$.
Because the Cauchy boundary  $\partial_C M$ of $M$ associated to $g'$
is almost polar, there is a sequence of
 the equilibrium potentials $e_n$ of $O_n \supset \partial_C M$ such that
$\cap_{n>1} O_{n} = \partial_C M$ and
\[ \| e_n \|_{H^1(M;d\mu_{g'})} \to 0 \mbox{ as $n\to\infty$}. \]
Then $u_n := (1-e_n) \psi u \in H^1_0 (M; d\mu_{g'})$ satisfies
\[ u_n \to \psi u  \mbox{ as $n\to\infty$}\]
in $H^1(M; d\mu_{g'})$,
and we get $\hat{u}_n = u_n + v_n \in H^1_0(M; d\mu_{g'})$ such that
\[ \hat{u}_n \to (1-\psi)u + \psi u = u   \mbox{ as $n\to\infty$}\]
in $H^1(M; d\mu_{g'})$.
\end{proof}

Next, we prove the inclusion of the 
essential spectrum and complete the proof of Theorem \ref{Th;1}.
We use the following characterization:
\begin{proposition}[Weyl's criterion] \label{Weyl's criterion}
A number $\lambda$ belongs to the essential spectrum of $\Delta$ if and only if
there is a sequence of orthonormal vectors $\{\phi_n\}$ of $L^2$ such that
\[ \| (\Delta-\lambda) \phi_n \|_{L^2} \to 0 \mbox{ as $n \to \infty$.} \] 
\end{proposition}

\begin{proof}[Proof of the inclusion of the 
essential spectrum.]
We assume (\ref{cond1}) and (\ref{cond2}) 
to prove the inclusion of the 
essential spectrum.
Hereafter, we denote $\De=\overline{\De}$ 
and
$\De'=\overline{\De'}$
because of their essential self-adjointness.
Let $\lambda \in \sigma_{\mbox{ess}}(\De)$
and $\phi_n \in D(\Delta)$ such that
\[ (\phi_i, \phi_j) =\delta_{ij}, \]
\begin{equation*}
\| (\De-\lambda) \phi_n \| \to 0 \mbox{ as $n\to\infty$},
\end{equation*}
where $\| \cdot \| = \sqrt{\left( \cdot,\cdot \right) }$.
Let $\chi$ be the function satisfying (\ref{cond1}) and
$\phi$ be the function defined as
\[ \phi = (1-\hat{r}/\ep)_+,\]
where $\hat{r}$ is the distance from the support of $\chi$.
Clearly, we have
\begin{equation} \label{Eq;4}
 \sup_{n>0} \| \phi \phi_n \| < \infty.
 \end{equation}
Moreover, taking into account that 
$\phi_n \in D(\De)$ implies
\[ \| \na \phi_n \|^2 = - (  \phi_n, \De \phi_n ), \]
it follows that
\begin{align*}
\limsup_{n \to \infty} \| \na (\phi \phi_n) \| 
&\le  \| \na \phi \|_{L^\infty} + \limsup_{n \to \infty} \| \na \phi_n\|   \\
&\le  \| \na \phi \|_{L^\infty} + \limsup_{n \to \infty} (\| \De \phi_n \| \|  \phi_n\| )^{1/2}< \infty.
\end{align*}
Hence,
\begin{equation} \label{eq;6}
\limsup_{n \to \infty} \| \phi \phi_n \|_{1} < \infty.
\end{equation}
Now,  specifying $\ep>0$ as in the statement,
by (\ref{eq;6}) and the fact that the embedding $H^1_0 (N) \subset L^2(N)$ is compact,
there exists a subsequence $\phi_{n(k)}$ of $\phi_n$ and $\phi' \in L^2$
such that
\[ \phi \phi_{n(k)} \to \phi' \mbox{ strongly in $L^2$ as $k \to \infty$}. \]
However, if $f \in L^2$, then $f \phi \in L^2$ and
\[ (f, \phi \phi_n) = (f\phi, \phi_n) \to 0 \mbox{ as $n\to\infty$}; \]
hence, $\phi \phi_{n(k)} \to 0$ weakly in $L^2$  as $k \to \infty$.
Because of the uniqueness of the weak-limits,
it follows that $\phi'=0$, and we may assume:
\begin{equation} \label{Eq;phi}
\| \phi \phi_n \| \to 0 \mbox{ as $n\to\infty$}
\end{equation}
without the loss of generality.
Since $\phi=1$ on $\rm{supp}(\chi)$,
\begin{align*}
& \| (\De-\lambda) (\chi \phi_n) \| \\
 & \le \| (\De \chi) \phi \phi_n \| + 2 \| 
(\na \chi, \na \phi_n) \| + \| \chi ( \De - \lambda) \phi_n \| \\
 & \le \| \De \chi \|_{L^\infty} \| \phi \phi_n \| + 2 \| 
\na \chi \|_{L^\infty} \|  \na \phi_n \|_{L^2({\rm{supp}(\chi)})} + 
 \| \chi\|_{L^\infty} \|( \De - \lambda) \phi_n \|.
\end{align*}
The first and third terms in the last line
tend to 0 as $n\to \infty$ because of (\ref{Eq;phi}) and (\ref{Eq;4}).
The second term can be estimated as
\[  \| \na \chi \|_{L^\infty} \|  \na (\phi \phi_n) \|
\le \| \na \chi \|_{L^\infty} 
\sqrt{ \| \De \phi_n \| \| \phi \phi_n \|} \to 0 \mbox{ as $n\to\infty$}. \]
Thus, since $1-\chi=0$ and $g=g'$ on $M \setminus K$,
\[  \| (\De'-\lambda) [(1-\chi) \phi_n] \|_{L^2(M;d\mu_{g'})}  
= \| (\De-\lambda) [(1-\chi) \phi_n] \|  \to 0 \mbox{ as $n\to\infty$}.\]
On the other hand,
\[
\| (1-\chi) \phi_n \|_{L^2(M;d\mu_{g'})} \ge | \| \phi_n \| - \|\chi \phi_n\| | \ge | 1- \| \phi\phi_n \|| \to 1 
\mbox{ as $n\to\infty$},
\]
and we conclude that 
$\lambda \in \sigma_{\mbox{ess}}(\De')$ by Weyl's criterion.
\end{proof}
Finally, we assume that $K$ is compact to prove the 
essential self-adjointness of the Laplacian and the stability of the 
essential spectrum; namely, Corollary \ref{corollary;1}.
\begin{proof}[Proof of Corollary \ref{corollary;1}.]
We will show that 
the Laplacian $\Delta'$ is essentially self-adjoint and that there 
exist the function $\chi$ and the subset $N$ of $M$
which satisfy conditions (\ref{cond1}) and (\ref{cond2}) for each metric $g$ and $g'$.
This will imply
 \[ \sigma_{\mbox{ess}}(\overline{\De}) =\sigma_{\mbox{ess}}(\overline{\De'})\]
by Theorem \ref{Th;1}.

Recall that if $K$ is compact, then its Cauchy boundary is empty so that
the Laplacian $\Delta'$ is essentially self-adjoint.

Since $K$ is compact and $g$ is smooth, there exists $\epsilon_0>0$
such that for any $0 < \epsilon < \epsilon_0$,
the metric $g$ and its higher order (up to 2nd) derivatives 
are bounded on $N=N(K;2\epsilon)=\{ x \in M:d(x,K) < 2\epsilon \}$.
Let
\[
\hat{\chi} (x) = (1 \wedge (2-3\tilde{r} (x) /  \epsilon))_+,
\]
where $\tilde{r}$ is the distance from $K$.
The function $\hat{\chi}$ is 1 on $N(K;\, \epsilon/3)$ and has support
in $N(K;\, 2\epsilon/3)$, and it satisfies
$\| \nabla \hat{\chi} \|_{L^\infty} \le 3/\epsilon$.
However, since $\hat{\chi}$ does not need to be in 
the Sobolev space $H^2$ of order $(2,2)$, we apply the
Friedrichs mollifier $j$ with radius $\delta>0$ for $\hat{\chi}$ to
find the smooth function $\chi=j * \hat{\chi} $.
If $\delta < \epsilon/3$, then $\chi$ satisfies:
\[
\begin{cases}
{\chi}(x)=1 \mbox{ for } x \in K; \\
\mbox{supp}(\chi) \subset N(K;\, 2\epsilon/3); \\ 
\| \nabla {\chi} \|_{L^\infty} \le 3/\epsilon; \\
\Delta \chi \in L^\infty; 
\end{cases}
\]
namely, condition (\ref{cond1}).
On the other hand, since $K$ is compact,  
$N_\ep$ and $N$ are relatively compact in $M$
with sufficiently small $\ep>0$.
Hence, $N$ has finite volume and finite diameter, and the 
Poincar\'e inequality holds on $N$, it follows that
the embedding $H^1_0 (N;\, d\mu_g) \subset L^2(N;\, d\mu_g)$ is compact;
that is, condition (\ref{cond2}).
We obtained the inclusion: 
$\sigma_{\mbox{ess}}(\overline{\De}) \subset \sigma_{\mbox{ess}}(\overline{\De'})$.

This argumentation holds true if we replace $g$ by $g'$ and we arrive at the conclusion.
\end{proof}

\section{Examples} \label{Examples}
In this section, we present examples 
of manifolds for which Theorem \ref{Th;1}
can be applied.

\begin{example}[see \cite{MasamuneRossman.99}]
Let $(M,g)$ be an $m$-dimensional complete Riemannian manifold,
$\Sigma \subset M$ be an $n$-dimensional compact manifold with $m \ge n+2$.
Assume that $M$ has the product structure $M^{m-n} \times M^n$ near $\Sigma$
and $g$ can be diagonalized.
Choose local coordinates
in a neighborhood $K$ of $\Sigma$ so that
\[
g=g_1 \oplus g_2
\]
in $K$, where $g_1$ is a metric on $M^{m- n}$
and $g_2$ is a metric on $M^n$.
Let $g'$ be another smooth metric on $M \setminus \Sigma$
so that 
\[
g' = 
\begin{cases}
f^2 g_1 \oplus g_2, &\mbox{ on $K$}, \\
g, &\mbox{ on $ M \setminus K $}.
\end{cases}
\]

If $m=2$, assume that $f \in L^{2 +\ep}(K; d\mu_g)$ for some $\ep \in (0,\infty)$.

If $m=3$, assume that inf$(f)>0$ and 
$f \in L^{(m(m-2)/2)+\ep} (K; d\mu_g)$ for some $\ep \in (0,\infty)$.

Then the manifold $M \setminus \Sigma$ with metrics $g$ and $g'$
satisfy the assumption of Theorem \ref{Th;1}.
In particular, if $M$ is compact, $\De'$ on ($M \setminus \Sigma, g')$
has discrete spectrum which satisfies the Weyl asymptotic formula \cite{MasamuneRossman.99}.
\end{example}

In the next example, manifold has fractal singularity.
\begin{example}
Let $(M,g)$ be a complete Riemannian manifold with dimension greater than 2.
Let $\Sigma \subset M$ be the Cantor set, $r$ the distance in $M$ from $\Sigma$,
$B(R)$ the $R$-neighborhood of $\Sigma$.
Set
\[g' = f^2 g,\]
where
\[
f(x) = 
\begin{cases}
r^{\ep}, & x \in B=B(1), \\
1, & x \in M \setminus B. 
\end{cases}
\]
It is shown in \cite{Masamune.05} that 
$\Sigma$ is the almost polar Cauchy boundary of $(M \setminus \Sigma; g' )$
if
\[ \ep > { { \ln2 - \ln3 } \over { 2 \ln 3 - \ln 2 } }.\]
The compact inclusion:
\[ H^1_0 (B \setminus \Sigma) \subset L^2(B\setminus \Sigma) \]
can be seen as follows.
By definition, 
$H^1_0 (B \setminus \Sigma) \subset H^1_0 (B)$,
and the inclusion: $H^1_0 (B) \subset L^2(B) 
= L^2(B\setminus \Sigma)$ is compact;
thus, it suffices to show 
\[ H^1_0 (B \setminus \Sigma) \supset H^1_0 (B).\]
Let $ u \in H^1_0 (B)$. Since $L^\infty \cap H^1_0(B) \subset H^1_0(B)$
is dense, we may assume that $u \in L^\infty$ without loss of generality.
Let $e_n$ be the equilibrium potential as in the proof of Theorem \ref{Th;1}.
Then $u_n = u (1-e_n) \in H^1_0 (B \setminus \Sigma)$ and
\[ u_n \to u \mbox { in $H^1_0 (B ; d\mu_g)$}, \]
and hence, $u \in H^1_0 (B \setminus \Sigma)$.
The function $\chi$ can be found as the relative equilibrium potential
of $B(1)$ and $B(2)$ applied the Friedrichs mollifier.
Therefore, 
$M \setminus \Sigma$ together with $g$ and $g'$ 
satisfy the condition of Theorem \ref{Th;1}. 
Let us point out that 
\begin{itemize}
\item $(M \setminus \Sigma,  g' )$ is $C^{1,1}$ and is not smooth, 
but Theorem \ref{Th;1} can be applied to this setting.
\item We can show the compactly imbedding 
$H^1_0(B\setminus \Sigma; d\mu_{g'}) \subset L^2(B;\setminus \Sigma; d\mu_{g'}) $
only for $\ep \ge 0$.
\end{itemize}
\end{example}

In the next example, $K$ has infinite volume with $g'$.
\begin{example}
Let $M$ be a 2-dimensional complete Riemannian manifold.
Delete a point $p \in M$ and set
\[g' = f^2 g,\]
where
\[
f(x) = 
\begin{cases}
r^{-\ep}, & x \in B=B(1), \\
1, & x \in M \setminus B,
\end{cases}
\]
and $r$ is the distance from $p$.
For any $\ep \ge 1$, $(M \setminus \{p\}, g')$ is complete and $\mu_{g'} (B \setminus \{p\}) = \infty$.

More generally, if $M$ is a complete manifold and
$\Sigma \subset M$ is a compact set,
then there is a smooth Riemannian metric
$g'$ on $M \setminus \Sigma$  and a compact set $K \subset M$ such that
$g = g'$ on $M \setminus K$, $(M \setminus \Sigma;g')$ is complete
and there exists a function $\chi$ and a subset $N$ of $M \setminus \Sigma$
satisfying conditions (\ref{cond1}) and  (\ref{cond2}), respectively.
\end{example}

\section{acknowledgments}
The author would like to thank Professor Hajime Urakawa 
for several advises which improved the original manuscript.

\end{document}